\documentclass[a4paper,reqno,11pt]{amsart}

\usepackage{amsmath, amsfonts, amssymb, amsthm, amscd}
\usepackage{graphicx}
\usepackage{url}
\usepackage{color}

\usepackage[utf8]{inputenc}
\usepackage[T1]{fontenc}
\usepackage{microtype}

\usepackage[a4paper,scale={0.72,0.74},marginratio={1:1},footskip=7mm,headsep=10mm]{geometry}

\usepackage{hyperref}

\numberwithin{equation}{section}

\newtheorem{thm}{Theorem}

\newtheorem{prop}[thm]{Proposition}

\newtheorem{rem}[thm]{Remark}

\newcommand{\bbP}{{\ensuremath{\mathbb P}} }

\newcommand{\cK}{{\ensuremath{\mathcal K}} }

\newcommand{\cP}{{\ensuremath{\mathcal P}} }
\newcommand{\cQ}{{\ensuremath{\mathcal Q}} }

\newcommand{\cZ}{{\ensuremath{\mathcal Z}} }

\newcommand{\gd}{\delta}

\newcommand{\gep}{\varepsilon}       

\renewcommand{\tilde}{\widetilde}          
\DeclareMathSymbol{\leqslant}{\mathalpha}{AMSa}{"36} 
\DeclareMathSymbol{\geqslant}{\mathalpha}{AMSa}{"3E} 
\DeclareMathSymbol{\eset}{\mathalpha}{AMSb}{"3F}     
\newcommand{\dd}{\text{\rm d}}             

\newcommand{\R}{\mathbb{R}}

\newcommand{\N}{\mathbb{N}}

\newcommand{\PEfont}{\mathrm}

\DeclareMathOperator{\p}{\ensuremath{\PEfont P}}
\DeclareMathOperator{\e}{\ensuremath{\PEfont E}}

\newcommand{\ind}{{\sf 1}}

\renewcommand{\epsilon}{\varepsilon} 
\renewcommand{\theta}{\vartheta} 
\renewcommand{\rho}{\varrho} 
\renewcommand{\phi}{\varphi}

\newenvironment{myenumerate}{%
\renewcommand{\theenumi}{{\rm(\arabic{enumi})}}%
\renewcommand{\labelenumi}{\theenumi}%
\begin{list}{\labelenumi}
	{%
	\setlength{\itemsep}{0.4em}%
	\setlength{\topsep}{0.5em}%
	\setlength\leftmargin{2.45em}%
	\setlength\labelwidth{2.05em}%
	\setlength{\labelsep}{0.4em}%
	\usecounter{enumi}%
	}%
	}%
{\end{list}
}

{\end{myenumerate}}

{\end{list}
}

\newenvironment{Cenumerate}{%
\renewcommand{\theenumi}{{\rm(C\arabic{enumi})}}%
\renewcommand{\labelenumi}{\theenumi}%
\begin{list}{\labelenumi}
	{%
	\setlength{\itemsep}{0.4em}%
	\setlength{\topsep}{0.5em}%
	\setlength\leftmargin{2.45em}%
	\setlength\labelwidth{2.05em}%
	\setlength{\labelsep}{0.4em}%
	\usecounter{enumi}%
	}%
	}%
{\end{list}
}

\newenvironment{myitemize}{%
\begin{list}{$\bullet$}%
 	{%
	\setlength{\itemsep}{0.4em}%
	\setlength{\topsep}{0.5em}%
	\setlength\leftmargin{2.45em}%
	\setlength\labelwidth{2.05em}%
	\setlength{\labelsep}{0.4em}%
	}%
	}%
{\end{list}}

\renewenvironment{itemize}{
\begin{myitemize}}%
{\end{myitemize}}

\title[Localization for $(\nabla + \Delta)$-pinning models]{Localization for
(1+1)-dimensional pinning models\\ with $(\nabla + \Delta)$-interaction}

\author{Martin Borecki}

\address{TU Berlin, Institut f\"ur Mathematik,
Strasse des 17. Juni 136, 10623 Berlin, Germany}

\email{borecki@math.tu-berlin.de}

\author{Francesco Caravenna}

\address{Dipartimento di Matematica Pura e Applicata, Universit\`a
degli Studi di Padova, via Trie\-ste 63, 35121 Padova, Italy}

\email{francesco.caravenna\@@math.unipd.it}



\date{\today}

\newcommand{\Pe}{\mathbb{P}_{\varepsilon,N}}
\newcommand{\Ze}{\mathcal{Z}_{\varepsilon,N}}

\DeclareMathOperator{\esssup}{ess\, sup}

\newcommand{\Leb}{\text{\sl Leb}}

\begin{document}

\begin{abstract}
We study the localization/delocalization phase transition in a class of
directed models for a homogeneous linear chain attracted to a defect line.
The self-interaction of the chain is of mixed gradient and Laplacian kind,
whereas the attraction to the defect line is of $\delta$-pinning type, with
strength $\epsilon \ge 0$. It is known that, when the self-interaction is purely Laplacian,
such models undergo a \emph{non-trivial} phase transition:
to localize the chain at the defect line,
the reward $\gep$ must be greater than a strictly positive critical threshold
$\gep_c > 0$. On the other hand, when the self-interaction is purely gradient,
it is known that the transition is \emph{trivial}: an arbitrarily small
reward $\gep > 0$ is sufficient to localize the chain at the defect line ($\gep_c = 0$).
In this note we show that in the mixed gradient and Laplacian case,
under minimal assumptions on the interaction potentials,
the transition is always trivial, that is $\gep_c = 0$.
\end{abstract}

\maketitle


\section{Introduction}

We consider a simple directed model for a homogeneous linear chain
$\{(i,\phi_i)\}_{0 \le i \le N}$, such as a polymer, which is randomly distributed 
in space and is attracted to the line
$\{(i,0)\}_{0 \le i \le N}$ through a \emph{pinning interaction}, see Figure~\ref{fig:pinning}. 
We will often refer to $\{\phi_i\}_{i}$ as the \emph{field}.
We discuss the localization properties of the model as a function of
the attraction strength $\varepsilon \geq 0$ and of the characteristics
of the chains, that are embodied in two potentials $V_1$ and $V_2$.

\begin{figure}[t]
\begin{center}
\includegraphics[width=.9\textwidth]{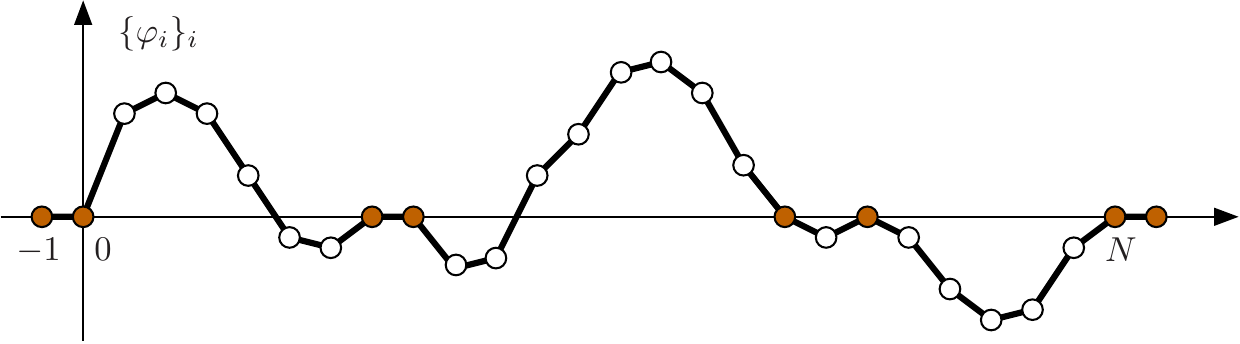}
\end{center}
\caption{\label{fig:pinning}A sample trajectory of the model $\bbP_{\gep, N}$.}
\end{figure}


\subsection{The model}

We first define the \emph{Hamiltonian}, which describes the
self-interaction of the field $\phi = \{\phi_i\}_i$:
\begin{equation} \label{eq:H}
	\mathcal{H}_{[-1,N+1]}(\phi) \;=\;
	\mathcal{H}_{[-1,N+1]}(\varphi_{-1},...,\varphi_{N+1}) \;:=\;
	\sum_{i=1}^{N+1}V_1(\nabla\varphi_i) \,+\,
	\sum_{i=0}^{N}V_2(\Delta\varphi_i) \,,
\end{equation}
where $N$ represents the length of the chain. The discrete gradient and Laplacian
of the field are defined respectively by $\nabla \phi_i := \phi_i - \phi_{i-1}$
and $\Delta \phi_i := \nabla \phi_{i+1} - \nabla \phi_i = \phi_{i+1} + \phi_{i-1} - 2 \phi_i$.
The precise assumptions on the potentials $V_1$ and $V_2$ are
stated below.

Given the strength of the pinning attraction $\varepsilon\geq 0$
between the chain and the defect line, we define
our model $\bbP_{\gep, N}$ as the following probability measure
on $\mathbb{R}^{N-1}$:
\begin{equation}\label{model}
    \Pe(\dd\varphi_1 \,, \ldots \,, \dd\varphi_{N-1})\ :=
    \frac{\exp(-\mathcal{H}_{[-1,N+1]}(\varphi))}{\Ze}\prod_{i=1}^{N-1}
    ( \varepsilon\delta_0(\dd\varphi_i) + \dd\varphi_i ) \,
\end{equation}
where we denote by $\delta_0(\cdot)$ 
the Dirac mass at zero, by $\dd\varphi_i = \Leb(\dd\phi_i)$ the Lebesgue measure on $\mathbb{R}$
and we choose for simplicity zero boundary conditions: 
$\varphi_{-1}=\varphi_0=\varphi_N=\varphi_{N+1}=0$
(see Figure~\ref{fig:pinning}).
The normalization constant $\Ze$ appearing in \eqref{model} plays an important role, as
we are going to see in a moment: it is called \emph{partition function} and is given by
\begin{equation} \label{eq:Z}
	\cZ_{\gep, N} \;=\; \int_{\R^{N-1}} e^{-\mathcal{H}_{[-1,N+1]}(\varphi)}
	\, \prod_{i=1}^{N-1} ( \varepsilon\delta_0(\dd\varphi_i) + \dd\varphi_i ) \,.
\end{equation}

We assume that the potentials $V_1,V_2:\mathbb{R}\rightarrow\mathbb{R}$
appearing in \eqref{eq:H} are measurable functions satisfying
the following conditions:
\begin{Cenumerate}
\item\label{C1} $V_1$ is bounded from below ($\inf_{x\in\R} V_1(x) > -\infty$),
symmetric ($V_1(x)=V_1(-x)$ for every $x\in\mathbb{R}$), such that
$\lim_{|x| \to \infty} V_1(x) = +\infty$ and $\int_\R e^{-2V_1(x)} \, \dd x < \infty$.

\item\label{C2} $V_2$ is bounded from below ($\inf_{x\in\R} V_2(x) > -\infty$),
bounded from above in a neighborhood of zero ($\sup_{|x| \le \gamma} V_2(x) < \infty$
for some $\gamma > 0$) and such that $\int_\R |x| \, e^{-V_2(x)} \, \dd x < \infty$.
\end{Cenumerate}
We stress that no continuity assumption is made. The symmetry of $V_1$
ensures that there is no ``local drift'' for the gradient of the field
(remarkably, no such assumption on $V_2$ is necessary; see
also Remark~\ref{th:remarkable} below). We point out that the hypothesis
that both $V_1$ and $V_2$ are finite everywhere could be relaxed, allowing them
to take the value $+\infty$ outside some interval $(-M,M)$, but we stick for simplicity
to the above stated assumptions.

\smallskip

The model $\bbP_{\gep, N}$ is an example of
a \emph{random polymer model}, more precisely a
(homogeneous) \emph{pinning model}. A lot of attention
has been devoted to this class of models in the recent mathematical 
literature (see~\cite{G,dH} for two beautiful monographs).

The main question, for models like ours, is whether
the pinning reward $\gep \ge 0$ is strong enough to localize the
field at the defect line for large $N$.
The case when the self-interaction of the field is of purely gradient type,
i.e., when $V_2 \equiv 0$ in \eqref{eq:H}, has been studied in depth
\cite{IY,CGZ,DGZ,BFO}, as well as the purely Laplacian case
when $V_1 \equiv 0$, cf.~\cite{CD,CD2}.
We now consider the mixed case when both $V_1 \not\equiv 0$
and $V_2 \not\equiv 0$, which is especially interesting from a physical viewpoint,
because of its direct relevance in modeling \emph{semiflexible polymers},
cf.~\cite{GKL}. Intuitively, the gradient interaction penalizes large elongations
of the chain while the Laplacian interaction penalizes curvature and bendings.



\subsection{Free energy and localization properties}

The standard way to capture the localization properties of models like ours
is to look at the exponential rate of growth (Laplace asymptotic behavior) as $N\rightarrow\infty$
of the partition function $\Ze$. More precisely, we define
the \emph{free energy} $F(\gep)$ of our model as
\begin{equation}\label{freeEnergy}
	F(\varepsilon) \,:=\, \lim_{N\rightarrow\infty}
	\frac{1}{N} \log \left( \frac{\Ze}{\cZ_{0,N}} \right) \,,
\end{equation}
where the limit is easily shown to exist by a standard super-additivity argument~\cite{G}. 

The function $\gep \mapsto \Ze$ is non-decreasing for fixed $N$
(cf. \eqref{eq:Z}), hence $\gep \mapsto F(\gep)$ is non-decreasing too. Recalling that
$F(0) = 0$, we define the \emph{critical value} $\gep_c$ as
\begin{equation} \label{eq:gepc}
	\gep_c \,:=\, \sup\{\gep \ge 0:\, F(\gep) = 0\} \,=\, \inf\{\gep \ge 0:\, F(\gep) > 0 \}
	\,\in\, [0,\infty] \,,
\end{equation}
and we say that our model $\{\bbP_{\gep, N}\}_{N\in\N}$ is
\begin{itemize}
\item \emph{delocalized} if $\gep < \gep_c$;
\item \emph{localized} if $\gep > \gep_c$.
\end{itemize}

This seemingly mysterious definition of localization and delocalization does
correspond to sharply different behaviors of the typical trajectories of our model.
More precisely, denoting by $\ell_N := \#\{1 \le i \le N-1:\, \phi_i = 0\}$
the number of contacts between the linear chain and the defect line, it is easily
shown by convexity arguments that 
\begin{itemize}
\item if $\varepsilon<\varepsilon_c$, for every $\gd >0$ there exists $c_\gd >0$ such that
\begin{equation}\label{chap0PathFallA}
	\Pe(\ell_N/N > \gd)\leq e^{-c_\gd N} \,, \qquad \text{for all }N\in\mathbb{N} \,;
\end{equation}

\item if $\varepsilon>\varepsilon_c$, there exists $\gd_\gep > 0$ and $c_\gep >0$ such that
\begin{equation}\label{chap0PathFallB}
 \Pe(\ell_N/N < \gd_\gep )\leq e^{-c_\gep N}\,, \qquad \text{for all }N\in\mathbb{N} \,.
\end{equation}
\end{itemize}
In words: if the model is delocalized then typically $\ell_N = o(N)$, while if
the model is localized then typically $\ell_N \ge \gd_\gep\, N$ with $\gd_\gep > 0$. 
We refer, e.g., to \cite[Appendix~A]{CD} for the proof of these facts.
We point out that the behavior of the model
at the critical point is a much more delicate issue,
which is linked to the regularity of the free energy.

\smallskip

Coming back to the critical value,
it is quite easy to show that $\gep_c < \infty$ (it is a by-product of our main result),
that is, the localized regime is non-empty. However,
it is not \emph{a priori} clear whether $\gep_c > 0$, i.e. 
whether the delocalized regime is non-empty.
For instance, in the purely Laplacian case ($V_1 \equiv 0$, cf.~\cite{CD}),
one has $\gep_c^\Delta > 0$.
On the other hand, in the purely gradient case ($V_2 \equiv 0$, cf.~\cite{BFO})
one has $\gep_c^\nabla = 0$ and the model is said to undergo a \emph{trivial phase transition}:
an arbitrarily small pinning reward is able to localize the linear chain.

The main result of this note is that in the general
case of mixed gradient and Laplacian interaction the phase transition
is always trivial.

\begin{thm}\label{Theorem}
For any choice of the potentials $V_1$, $V_2$ satisfying assumptions \ref{C1} and
\ref{C2} one has $\gep_c = 0$, i.e., $F(\gep) > 0$ for every $\gep > 0$.
\end{thm}

Generally speaking, it may be expected that the gradient interaction terms should dominate over
the Laplacian ones, at least when $V_1$ and $V_2$ are comparable functions.
Therefore, having just recalled that $\gep_c^\nabla = 0$, Theorem~\ref{Theorem}
does not come as a surprise. Nevertheless, our assumptions \ref{C1} and \ref{C2}
are very general and allow for strikingly different asymptotic behavior of the potentials:
for instance, one could choose $V_1$ to grow only logarithmically and $V_2$
exponentially fast (or even more). The fact that the gradient interaction dominates
even in such extreme cases is quite remarkable.

\begin{rem}\rm
Our proof yields actually an explicit lower bound on the free energy, which is however
quite poor. This issue is discussed in detail
in Remark~\ref{th:improving} in section~\ref{sec:lb} below.
\end{rem}

\begin{rem}\rm
Theorem~\ref{Theorem} was first proved in the Ph.D. thesis \cite{B} in the special case
when both the interaction potentials are quadratic: 
$V_1(x) = \frac{\alpha}{2}\,x^2$ and $V_2(x) = \frac{\beta}{2}\,x^2$, for any
$\alpha, \beta > 0$.
We point out that, with such a choice for the potentials, 
the free model $\bbP_{0,N}$ is a Gaussian law
and several explicit computations are possible.
\end{rem}

\subsection{Organization of the paper}

The rest of the paper is devoted to the proof of the Theorem~\ref{Theorem}, which is organized
in two parts:
\begin{itemize}
\item in section~\ref{sec:Markov} we give a basic representation of the free model ($\gep = 0$)
as the bridge of an integrated Markov chain, and we study some asymptotic properties of
this hidden Markov chain;
\item in section~\ref{sec:lb} we give an explicit lower bound on the partition function
$\cZ_{\gep, N}$ which, together with the estimates obtained in section~\ref{sec:Markov}, yields
the positivity of the free energy $F(\gep)$ for every $\gep > 0$, hence the proof of
Theorem~\ref{Theorem}.
\end{itemize}
Some more technical points are deferred to Appendix~\ref{sec:Harris}.


\subsection{Some recurrent notation and basic results}

We set $\R^+ = [0,\infty)$, $\N := \{1,2,3,\ldots\}$ and $\N_0 := \N \cup \{0\} = \{0,1,2,\ldots\}$.
We denote by $\Leb$ the Lebesgue measure on $\R$.

We denote by $L^p(\R)$, for $p \in [1,\infty]$,
the Banach space of (equivalence classes of) measurable functions
$f: \R \to \R$ such that $\|f\|_p < \infty$, where
$\|f\|_p := (\int_\R |f(x)|^p \, \dd x)^{1/p}$ for $p \in [1,\infty)$
and $\|f\|_\infty := \esssup_{x\in\R} |f(x)|
= \inf\{M>0:\, \Leb\{x \in \R:\, |f(x)|>M\}=0\}$.

Given two measurable functions $f, g: \R \to \R^+$, their convolution
is denoted as usual by $(f * g)(x) := \int_\R f(x-y) \, g(y) \, \dd y$.
We recall that if $f \in L^1(\R)$ and $g \in L^\infty(\R)$ then
$f * g$ is bounded and continuous, cf. Theorem~D.4.3 in~\cite{cf:MT}.


\medskip

\section{A Markov chain viewpoint}

\label{sec:Markov}

We are going to construct a Markov chain which will be the basis of our analysis.
Consider the linear integral operator $f \mapsto \cK f$ defined
(for a suitable class of functions $f$) by
\begin{equation} \label{eq:K}
 (\mathcal{K}f)(x) := \int_\R k(x,y)f(y)\, \dd y \,, \qquad
 \text{where} \quad  k(x,y) := e^{-V_1(y)-V_2(y-x)} \,.
\end{equation}
The idea is to modify $k(x,y)$ with boundary terms to make $\cK$ a probability kernel.


\subsection{Integrated Markov chain}

By assumption \ref{C1} we have $\| e^{-2 V_1} \|_1 < \infty$.
It also follows by assumption \ref{C2} that $e^{-V_2} \in L^1(\R)$,
because we can write
\begin{equation*}
	\| e^{-V_2} \|_1 \;=\; \int_\R e^{-V_2(x)} \,\dd x \;\le\; 2 \sup_{x \in [-1,1]} e^{-V_2(x)} 
	\,+\, \int_{\R \setminus [-1,1]} |x| \, e^{-V_2(x)} \, \dd x \;<\; \infty \,.
\end{equation*}
Since we also have $e^{-V_2} \in L^\infty(\R)$, again by \ref{C2}, it follows that
$e^{-V_2} \in L^p(\R)$ for all $p \in [1,\infty]$, in particular
$\| e^{-2 V_2} \|_1 < \infty$. We then obtain
\begin{equation*}
	\int_{\R\times\R} k(x,y)^2 \, \dd x \, \dd y \;=\;
	\int_\R e^{-2 V_1(y)} \left( \int_\R e^{-2 V_2(y-x)} \, \dd x \right) \dd y \;=\;
	\| e^{-2 V_1} \|_1 \, \| e^{-2 V_2} \|_1 \;<\; \infty \,.
\end{equation*}
This means that $\cK$ is Hilbert-Schmidt, hence a compact operator on $L^2(\R)$.
Since $k(x,y) \ge 0$ for all $x,y \in \R$, we can then apply an infinite dimensional
version of the celebrated Perron-Frobenius Theorem.
More precisely, Theorem~1 in~\cite{Z} ensures that the spectral radius $\lambda>0$ of
$\mathcal{K}$ is an isolated eigenvalue, with corresponding right and left
eigenfunctions $v,w\in L^2(\mathbb{R})$ satisfying $w(x) > 0$ and $v(x)>0$ for almost
every $x\in\R$:
\begin{equation} \label{eq:eigeneq}
   v(x) \;=\; \frac{1}{\lambda}\int_{\mathbb{R}}k(x,y)\,v(y)\, \dd y \,,
   \qquad
   w(x) \;=\; \frac{1}{\lambda}\int_{\mathbb{R}}w(y)\, k(y,x)\,\dd y \,.
\end{equation}
These equations give a canonical definition of $v(x)$ and $w(x)$ for every
$x\in\R$. Since $k(x,y) > 0$ for all $x,y \in \R$,
it is then clear that $w(x)>0$ and $v(x)>0$ \emph{for every} $x\in\mathbb{R}$.

We can now define a probability kernel $\cP(x,\dd y)$ by setting
\begin{equation} \label{eq:p}
	\cP(x,\dd y) \;:=\; p(x,y) \,\dd y
	\;:=\; \frac{1}{\lambda} \, \frac{1}{v(x)} \, k(x,y) \, v(y) \,.
\end{equation}
Since $\cP(x, \R) = \int_\R p(x,y) \, \dd y = 1$ for every $x \in \R$,
we can define a Markov chain on $\R$ with transition kernel $\cP(x, \dd y)$.
More precisely, for $a,b\in\mathbb{R}$ let $(\Omega,\mathcal{A},\p^{(a,b)})$ be a probability space on which is defined a Markov chain $Y = \{Y_i\}_{i\in\N_0}$
on $\R$ such that
\begin{equation}\label{eq:Y}
	Y_0 \;=\; a \,, \qquad
	\p^{(a,b)}(Y_{n+1} \in \dd y \,|\, Y_n=x) \;=\;
    \cP(x, \dd y) \,,
\end{equation}
and we define the corresponding \emph{integrated Markov chain} $W = \{W_i\}_{i\in\N_0}$ setting
\begin{equation} \label{eq:W}
	W_0 \;=\; b \,, \qquad W_n \;=\; b + Y_1 + \ldots + Y_n \,.
\end{equation}
The reason for introducing such processes is that they are closely related
to our model, as we show in Proposition~\ref{PropFreeModel} below.
We first need to compute explicitly
the finite dimensional distributions of the process $W$.

\begin{prop}\label{PropLawOfW}
For every $n\in\mathbb{N}$, setting $w_{-1}:=b-a$ and $w_0:=b$, we have
\begin{equation} \label{eq:PropLawOfW}
    \p^{(a,b)}\left((W_1,...,W_n) \in (\dd w_1,...,\dd w_n) \right)
    \;=\; \frac{v(w_n-w_{n-1})}{\lambda^n\,v(a)} \,
    e^{-\mathcal{H}_{[-1,n]}(w_{-1},...,w_n)}
    \prod_{i=1}^n \dd w_i \,.
\end{equation}
\end{prop}

\begin{proof}
Since $Y_i=W_i-W_{i-1}$ for all $i\geq 1$, the law of $(W_1,...,W_n)$ is determined by 
the law of $(Y_1,...,Y_n)$.
If we set $y_i:=w_i-w_{i-1}$ for $i\geq 2$ and $y_1:=w_1-b$, it then
suffices to show that the right hand side of equation \eqref{eq:PropLawOfW}
is a probability measure under which the variables $(y_i)_{i=1,\ldots,n}$
are distributed like the first $n$ steps of a Markov chain starting at $a$ with
transition kernel $p(x,y)$. To this purpose, the Hamiltonian can be rewritten as
\begin{equation*}
	\mathcal{H}_{[-1,n]}(w_{-1},...,w_n)
	\;=\; V_1(y_1) + V_2(y_{1}-a) +
	\sum_{i=2}^n \big( V_1(y_i) + V_2(y_{i}-y_{i-1}) \big) \,.
\end{equation*}
Therefore, recalling the definitions
\eqref{eq:K} of $k(x,y)$ and \eqref{eq:p} of $p(x,y)$, we can write
\begin{align*}
       \frac{v(w_n-w_{n-1})}{\lambda^n \, v(a)} e^{-\mathcal{H}_{[-1,n]}(w_{-1},...,w_n)}
        & \;=\; \frac{v(y_n)}{\lambda^n \, v(a)}\,k(a,y_1) \prod_{i=2}^n k(y_{i-1},y_i) \\
        & \;=\; p(a,y_1) \, \prod_{i=2}^n p(y_{i-1},y_i) \,,
\end{align*}
which is precisely the density of $(Y_1,...,Y_n)$ under $\p^{(a,b)}$ with respect to the
Lebesgue measure $\dd y_1\cdots \dd y_n$. Since the map from $(w_i)_{i=1,\ldots,n}$
to $(y_i)_{i=1,\ldots,n}$ is linear with determinant one, the proof is completed.
\end{proof}

For $n\geq 2$ we denote by $\varphi_{n}^{(a,b)}(\cdot,\cdot)$ the density of
the random vector $(W_{n-1},W_{n})$:
\begin{equation} \label{eq:phi}
	\varphi_{n}^{(a,b)}(w_1,w_2)
	\;:=\; \frac{\p^{(a,b)}\left( (W_{n-1},W_n) \in (\dd w_1,\dd w_2) \right)}
	{\dd w_1\dd w_2} \,, \qquad \text{for } w_1, w_2 \in \R \,.
\end{equation}
We can now show that our model $\bbP_{\gep, N}$ in the free case, that is for $\gep = 0$,
is nothing but a bridge of the integrated Markov chain $W$.

\begin{prop}\label{PropFreeModel}
For every $N\in\N$ the following relations hold:
\begin{align}
	\label{eq:bridgeP}
	\mathbb{P}_{0,N}(.) &\;=\; \p^{(0,0)} \big( \, (W_1,...,W_{N-1})\in \,\cdot\, \big|
    \, W_N=W_{N+1}=0 \, \big) \,, \\
    \label{eq:bridgeZ}
    \mathcal{Z}_{0,N} &\;=\; \lambda^{N+1}\,\varphi_{N+1}^{(0,0)}(0,0) \,.
\end{align}
\end{prop}
\begin{proof}
By Proposition \ref{PropLawOfW}, for every measurable subset $A \subseteq \R^{N-1}$
we can write
\begin{equation} \label{eq:interm}
\begin{split}
	& \p^{(0,0)} \big( (W_1,...,W_{N-1}) \in A \,\big|\,
	W_N=W_{N+1}=0 \big) \\
	& \qquad \quad \;=\; \frac{1}{\lambda^{N+1}\,\varphi_{N+1}^{(0,0)}(0,0)} \; \int_A
	e^{-\mathcal{H}_{[-1,N+1]}(w_{-1},...,w_{N+1})}\prod_{i=1}^{N-1}\dd w_i \,,
\end{split}
\end{equation}
where we set $w_{-1}=w_0=w_N=w_{N+1}=0$. Choosing $A = \R^{N-1}$ and recalling
the definition \eqref{eq:Z} of the partition function $\cZ_{\gep, N}$, we obtain
relation \eqref{eq:bridgeZ}. Recalling the definition \eqref{model} of our model
$\bbP_{\gep, N}$ for $\gep = 0$, we then see that \eqref{eq:interm} is nothing
but \eqref{eq:bridgeP}.
\end{proof}


\subsection{Some asymptotic properties}
\label{sec:asymp}

We now discuss some basic properties of the Markov chain
$Y = \{Y_i\}_{i\in\N_0}$, defined in \eqref{eq:Y}. We recall that the underlying probability
measure is denoted by $\p^{(a,b)}$ and we have $a = Y_0$.
The parameter $b$ denotes the starting point $W_0$ of the integrated Markov
chain $W = \{W_i\}_{i\in\N_0}$ and is irrelevant for the study of $Y$,
hence we mainly work under $\p^{(a,0)}$.

\smallskip

Since $p(x,y) > 0$ for all $x,y \in \R$, cf. \eqref{eq:p}
and \eqref{eq:K}, the Markov chain $Y$ is \emph{$\phi$-irreducible}
with $\phi = \Leb$:
this means (cf.~\cite[\S4.2]{cf:MT}) that for every measurable subset $A \subseteq \R$
with $\Leb(A) > 0$ and for every $a \in \R$ there exists
$n \in \N$, possibly depending on $a$ and $A$, such that $\p^{(a,0)}(Y_n \in A) > 0$.
In our case we can take $n=1$, hence the chain $Y$ is also \emph{aperiodic}.

Next we observe that $\int_\R v(x) \, w(x) \, \dd x \le \|v\|_2 \, \|w\|_2 < \infty$,
because $v ,w \in L^2(\R)$ by construction.
Therefore we can define the probability measure $\pi$ on $\R$ by
\begin{equation} \label{eq:pi}
	\pi(\dd x) \;:=\; \frac{1}{c} \, v(x) \, w(x) \, \dd x \,, \qquad
	\text{where} \qquad c := \int_\R v(x) \, w(x) \, \dd x \,.
\end{equation}
The crucial observation is that $\pi$ is an invariant probability
for the transition kernel $\cP(x,\dd y)$: from \eqref{eq:p} and \eqref{eq:eigeneq} we have
\begin{equation} \label{eq:invpi}
\begin{split}
	\int_{x \in \R} \pi(\dd x) \, \cP(x, \dd y)
	& \;=\; \int_{x \in \R} 	\frac{v(x) \, w(x)}{c} \,\dd x \,
	\frac{k(x,y) \, v(y)}{\lambda \, v(x)} \, \dd y \\
	& \;=\; \frac{w(y) \, v(y)}{c} \, \dd y \;=\; \pi(\dd y) \,.
\end{split}
\end{equation}
Being $\phi$-irreducible and admitting an invariant probability measure,
the Markov chain $Y = \{Y_i\}_{i\in\N_0}$ is \emph{positive recurrent}.
For completeness, we point out that
$Y$ is also Harris recurrent, hence it is a \emph{positive Harris chain}, cf. \cite[\S10.1]{cf:MT},
as we prove in Appendix~\ref{sec:Harris}
(where we also show that $\Leb$ is a maximal irreducibility measure for $Y$).

\smallskip

Next we observe that the right eigenfunction \emph{$v$ is bounded and continuous}:
in fact, spelling out the first relation in \eqref{eq:eigeneq}, we have
\begin{equation} \label{eq:vcont}
	v(x) \;=\; \frac{1}{\lambda} \int_\R e^{-V_2(y-x)} \, e^{-V_1(y)} \, v(y) \, \dd y
	\;=\; \frac{1}{\lambda} \big( e^{-V_2} * (e^{-V_1} \, v) \big)(x) \,.
\end{equation}
By construction $v \in L^2(\R)$ and by assumption \ref{C1}
$e^{-V_1} \in L^2(\R)$, hence
$(e^{-V_1} \, v) \in L^1(\R)$. Since $e^{-V_2} \in L^\infty(\R)$ by assumption \ref{C2},
it follows by \eqref{eq:vcont} that $v$, being
the convolution of a function in $L^\infty(\R)$ with a function in $L^1(\R)$,
is bounded and continuous. In particular, $\inf_{|x|\le M} v(x) > 0$ for every $M > 0$,
because $v(x) > 0$ for every $x \in \R$, as we have already
remarked (and as it is clear from \eqref{eq:vcont}).

\smallskip

Next we prove a suitable \emph{drift condition} on the kernel $\cP$.
Consider the function
\begin{equation} \label{eq:U}
	U(x) \;:=\; \frac{|x| \, e^{V_1(x)}}{v(x)} \,,
\end{equation}
and note that
\begin{equation} \label{eq:cpu}
\begin{split}
	(\cP U)(x) & \;=\; \int_\R p(x,y) \, U(y) \, \dd y \;=\;
	\frac{1}{\lambda\, v(x)} \int_\R e^{-V_2(y-x)} \, |y| \,\dd y \\
	& \;=\; \frac{1}{\lambda\, v(x)} \int_\R e^{-V_2(z)} \, |z+x| \,\dd z
	\;\le\; \frac{c_0 \,+\, c_1 \, |x|}{\lambda\, v(x)} \,,
\end{split}
\end{equation}
where $c_0 := \int_\R |z| \, e^{-V_2(z)} \, \dd z < \infty$ and
$c_1 := \int_\R e^{-V_2(z)} \, \dd z < \infty$
by our assumption \ref{C2}. Then we fix $M \in (0,\infty)$ such that
\begin{equation*}
	U(x) \,-\, (\cP U)(x) \;=\;
	\frac{|x| \, e^{V_1(x)}}{v(x)} \,-\,
	\frac{c_1 \, |x| \,+\, c_0}{\lambda \, v(x)}
	\;\ge\; \frac{1 + |x|}{v(x)} \,, \qquad
	\text{for } |x| > M \,.
\end{equation*}
This is possible because $V_1(x) \to \infty$ as $|x| \to \infty$,
by assumption \ref{C1}. 	Next we observe that
\begin{equation*}
	b \;:=\; \sup_{|x| \le M} \big( (\cP U)(x) - U(x) \big) \;<\; \infty \,,
\end{equation*}
as it follows from \eqref{eq:U} and \eqref{eq:cpu} recalling that $v$
is bounded and $\inf_{|x|\le M} v(x) > 0$ for all $M>0$. Putting together these
estimates, we have shown in particular that
\begin{equation} \label{eq:drift}
	(\cP U)(x) \,-\, U(x) \;\le\; -\frac{1 + |x|}{v(x)} \,+\, b \, \ind_{[-M,M]}(x) \,.
\end{equation}
This relation is interesting because it allows to prove the following result.

\begin{prop}\label{PropBoundOnY}
There exists a constant $C \in (0,\infty)$
such that for all $n\in\mathbb{N}$ we have
\begin{equation} \label{eq:wow}
	\e^{(0,0)} \big( |Y_n| \big) \,\le\, C \,, \qquad
	\e^{(0,0)} \left( \frac{1}{v(Y_n)} \right) \,\le\, C \,.
\end{equation}
\end{prop}
\begin{proof}
In Appendix~\ref{sec:Harris} we prove that $Y = \{Y_i\}_{i\in\N_0}$
is a $T$-chain (see Chapter~6 in~\cite{cf:MT} for the definition of $T$-chains).
It follows by Theorem~6.0.1 in~\cite{cf:MT}
that for irreducible $T$-chains every compact set is petite
(see \S5.5.2 in~\cite{cf:MT} for the definition of petiteness).
We can therefore apply Theorem~14.0.1 in~\cite{cf:MT}: relation
\eqref{eq:drift} shows that condition (iii) in that theorem is satisfied
by the function $U$. Since $U(x) < \infty$ for every $x \in \R$, this
implies that \emph{for every starting point $x_0 \in \R$}
and for every measurable function $g : \R \to \R$ with 
$|g(x)| \le (const.) (1+|x|)/v(x)$ we have
\begin{equation}\label{eq:limitpi}
	\lim_{n\to\infty} \e^{(x_0,0)} \big( g(Y_n) \big) \;=\;
	\int_\R g(z) \, \pi(\dd z) \;<\; \infty \,.
\end{equation}
The relations in \eqref{eq:wow} are obtained by
taking $x_0 = 0$ and $g(x) = |x|$ or $g(x) = 1/v(x)$.
\end{proof}

As a particular case of \eqref{eq:limitpi}, we observe that for
every measurable subset $A \subseteq \R$ and for every $x_0 \in \R$
we have
\begin{equation} \label{eq:ergth}
	\lim_{n\to\infty} \p^{(x_0, 0)}(Y_n \in A) \;=\; \pi(A)
	\;=\; \frac 1c \int_A v(x) \, w(x) \, \dd x \,.
\end{equation}
This is actually a consequence of the classical ergodic theorem for aperiodic Harris
recurrent Markov chains, cf. Theorem~113.0.1 in~\cite{cf:MT}.

\begin{rem}\rm
\label{th:remarkable}
Although we do not use this fact explicitly, it is interesting to observe
that the invariant probability $\pi$ is symmetric. To show this,
we set $\tilde v(x) := e^{-V_1(x)} v(-x)$ and we note
that by the first relation in \eqref{eq:eigeneq},
with the change of variables $y \mapsto -y$, we can write
\begin{equation*}
	\tilde v(x) \;=\; \frac{1}{\lambda} \int_\R e^{-V_1(x)} \,
	k(-x,y) \, v(y) \, \dd y \;=\; \frac{1}{\lambda} \int_\R e^{-V_1(x)} \,
	k(-x,-y) \, e^{V_1(y)} \, \tilde v(y) \, \dd y \,.
\end{equation*}
However $e^{-V_1(x)} \, k(-x,-y) \, e^{V_1(y)} = k(y,x)$,
as it follows by \eqref{eq:K} and the symmetry of $V_1$
(recall our assumption \ref{C1}). Therefore $\tilde v$
satisfies the same functional equation
$\tilde v(x) = \frac{1}{\lambda} \int_\R \tilde v(y) \, k(y,x) \, \dd y$
as the right eigenfunction $w$,
cf. the second relation in \eqref{eq:eigeneq}.
Since the right eigenfunction is uniquely determined up to constant multiples, there must
exist $C > 0$ such that $w(x) = C \, \tilde v(x)$ for all $x\in\R$. Recalling
\eqref{eq:pi}, we can then write
\begin{equation}
	\pi(\dd x) \,=\, \frac{1}{\tilde c} \, 
	e^{-V_1(x)} \, v(x) \, v(-x) \, \dd x \,, \qquad \ \
	\tilde c \,:=\, \frac{c}{C}
	\,,
\end{equation}
from which the symmetry of $\pi$ is evident.

From the symmetry of $\pi$ and \eqref{eq:limitpi} it follows that
$\e^{(0,0)}(Y_n) \to 0$ as $n\to\infty$, whence
the integrated Markov chain $W = \{W_i\}_{i\in\N_0}$ is somewhat close
to a random walk with zero-mean increments. We stress that this follows
by the symmetry of $V_1$, without the need of an analogous requirement on $V_2$.
\end{rem}



\subsection{Some bounds on the density}

We close this section with some bounds on the behavior of the density
$\phi_n^{(0,0)}(x,y)$ at $(x,y) = (0,0)$.

\begin{prop}\label{PropLowerUpperBound}
There exist positive constants $C_1, C_2$ such that for all odd $N\in\N$
\begin{equation} \label{eq:boundphi}
	\frac{C_1}{N} \;\le\; \phi_N^{(0,0)}(0,0) \;\le\; C_2 \,.
\end{equation}
\end{prop}

\noindent
The restriction to odd values of $N$ is just for technical convenience.
We point out that neither of the bounds in \eqref{eq:boundphi} is sharp, as the conjectured
behavior (in analogy with the pure gradient case,
cf.~\cite{CGZ}) is $\phi_N^{(0,0)}(0,0) \sim (const.) \, N^{-1/2}$.



\begin{proof}[Proof of Proposition~\ref{PropLowerUpperBound}]

We start with the lower bound.
By Proposition \ref{PropFreeModel} and equation \eqref{eq:Z}, we have
\[
	\varphi_{2N+1}^{(0,0)}(0,0) \;=\; \frac{1}{\lambda^{2N+1}} \, \cZ_{0,2N}
	\;=\; \frac{1}{\lambda^{2N+1}}\int_{\mathbb{R}^{2N-1}}
	e^{-\sum_{i=1}^{2N+1} V_1(\nabla\varphi_i)-
	\sum_{i=0}^{2N} V_2(\Delta\varphi_i)}\,\prod_{i=1}^{2N-1}\,\dd\varphi_i \,,
\]
where we recall that
the boundary conditions are $\varphi_{-1}=\varphi_{0}=\varphi_{2N}=\varphi_{2N+1}=0$.
To get a lower bound, we restrict the integration on the set
\begin{equation*}
	C_N^1 \;:=\; \left\{ (\phi_1, \ldots, \phi_{2N-1}) \in \R^{2N-1}:\
	|\varphi_{N}-\varphi_{N-1}| < \frac \gamma 2\ ,\ |\varphi_{N}-\varphi_{N+1}| <
	\frac \gamma 2 \right\} \,,
\end{equation*}
where $\gamma > 0$ is the same as in assumption \ref{C2}.
On $C_N^1$ we have $|\nabla\varphi_{N+1}|<\gamma/2$ 
and $|\Delta\varphi_{N}|<\gamma$, therefore
$V_2(\Delta\phi_{N}) \le M_\gamma := \sup_{|x| \le \gamma} V_2(x) < \infty$.
Also note that $V_{1}(\nabla\phi_{2N+1}) = V_1(0)$ due to the boundary conditions.
By the symmetry of $V_1$ (recall assumption \ref{C1}),
setting $C_N^2(\phi_N) := \{ (\phi_1, \ldots, \phi_{N-1}) \in \R^{N-1}:\
|\varphi_{N}-\varphi_{N-1}| < \gamma/2\}$,
we can write
\begin{align*}
 	& \varphi_{2N+1}^{(0,0)}(0,0) \\
	& \;\geq\; \frac{e^{-(M_\gamma + V_1(0))}}{\lambda^{2N+1}}\int_{C_N^1}
	e^{-\sum_{i=1}^{N} V_1(\nabla\varphi_i)-\sum_{i=0}^{N-1} V_2(\Delta\varphi_i)}
	\ e^{-\sum_{i=N+1}^{2N+1} V_1(\nabla\varphi_i)-\sum_{i=N+1}^{2N} V_2(\Delta\varphi_i)}
	\prod_{i=1}^{2N-1} \dd\varphi_i \\
	& \;=\; \frac{e^{-(M_\gamma + V_1(0))}}{\lambda^{2N+1}} \int_{\mathbb{R}}
	\dd\varphi_N \left[\int_{C_N^2(\phi_N)}
	e^{-\sum_{i=1}^{N} V_1(\nabla\varphi_i)-\sum_{i=0}^{N-1} V_2(\Delta\varphi_i)}
	\,\prod_{i=1}^{N-1}\,\dd\varphi_i\right]^2 \,.
\end{align*}
For a given $c_N>0$, we restrict the integration over $\phi_N \in [-c_N, c_N]$
and we apply Jensen's inequality, getting
\begin{align}\label{lowerBoundPhi}
 	\varphi_{2N+1}^{(0,0)}(0,0) & \;\geq\; \frac{e^{-(M_\gamma + V_1(0))}}{\lambda \cdot 2c_N}
	\left[\frac{1}{\lambda^N}\int_{-c_N}^{c_N}\,\dd\varphi_N\int_{C_N^2(\phi_N)}
	e^{-\sum_{i=1}^{N} V_1(\nabla\varphi_i)-\sum_{i=0}^{N-1} V_2(\Delta\varphi_i)}
	\,\prod_{i=1}^{N-1}\,\dd\varphi_i\right]^2\nonumber \\
	& \;\ge\; \frac{e^{-(M_\gamma + V_1(0))}}{\lambda \cdot 2c_N} \,
	\frac{v(0)}{\|v\|_\infty}
	\left[\frac{1}{\lambda^N}\int_{-c_N}^{c_N}\,\dd\varphi_N\int_{C_N^2(\phi_N)}
	\frac{v(\varphi_{N}-\varphi_{N-1})}{v(0)}\right.\nonumber\\
	&\qquad\qquad\qquad\qquad\quad\qquad\qquad\left. \cdot\ e^{-\sum_{i=1}^{N}
	V_1(\nabla\varphi_i)-\sum_{i=0}^{N-1} V_2(\Delta\varphi_i)}
	\,\prod_{i=1}^{N-1}\,\dd\varphi_i\right]^2\nonumber\\
	& \;=\; \frac{e^{-(M_\gamma + V_1(0))}}{\lambda \cdot 2c_N} \,
	\frac{v(0)}{\|v\|_\infty}
	\left[\p^{(0,0)}(|W_N|\leq c_N\ ,\ |W_N-W_{N-1}|\leq\gamma/2)\right]^2 \,,
\end{align}
where in the last equality we have used Proposition~\ref{PropLawOfW}. Now we observe that
\begin{align} \label{eq:probest}
\begin{split}
	\p^{(0,0)}(|W_N|\leq c_N\ ,\ |Y_N| \leq \gamma/2) &\;\ge\;
	1 - \p^{(0,0)}(|W_N|> c_N) - \p^{(0,0)}(|Y_N| > \gamma/2) \\
	& \;\geq\; 1 - \frac{1}{c_N} \, \e^{(0,0)}[|W_N|]
	- \p^{(0,0)}(|Y_N| > \gamma/2) \,.
\end{split}
\end{align}
By \eqref{eq:ergth}, as $N\to\infty$ we have $\p^{(0,0)}(|Y_N| > \gamma/2) \to
\pi( \R \setminus (-\frac\gamma 2, \frac \gamma 2)) =: 1 - 3\eta$,
with $\eta > 0$, therefore $\p^{(0,0)}(|Y_N| > \gamma/2) \le 1- 2\eta$
for $N$ large enough. On the other hand, by Proposition~\ref{PropBoundOnY} we have
\begin{equation}\label{GenPotInacurracy}
	\e^{(0,0)}[|W_N|] \;\leq\; \sum_{n=1}^N \e^{(0,0)}[|Y_n|]
	\;\leq\; C \, N \,.
\end{equation}
If we choose $c_N := C N / \eta$, from \eqref{lowerBoundPhi}, \eqref{eq:probest}
and \eqref{GenPotInacurracy} we obtain
\begin{equation*}
	\varphi_{2N+1}^{(0,0)}(0,0) \;\ge\;
	\frac{e^{-(M_\gamma + V_1(0))}}{2\lambda C} \, \frac{v(0)}{\|v\|_\infty}
	\, \eta^3 \, \frac{1}{N} \;=\; \frac{(const.)}{N} \,,
\end{equation*}
which is the desired lower bound in \eqref{eq:boundphi}.

\smallskip

The upper bound is easier.
By assumptions \ref{C1} and \ref{C2} both $V_1$ and $V_2$ are bounded
from below, therefore we can replace $V_1(\nabla \phi_{2N+1})$,
$V_1(\nabla \phi_{2N})$, $V_2(\Delta \phi_{2N})$
and $V_2(\Delta \phi_{2N-1})$ by the constant $\tilde c :=
\inf_{x\in\R} \min\{V_1(x), V_2(x)\} \in \R$ getting the upper bound:
\begin{align*}
	\varphi^{(0,0)}_{2N+1}(0,0) & \;=\; \frac{1}{\lambda^{2N+1}}
	\int_{\mathbb{R}^{2N-1}} e^{-\mathcal{H}_{[-1,2N+1]}(\varphi)} 
	\,\prod_{i=1}^{2N-1}\,\dd\varphi_i \\
	& \;\leq\; \frac{e^{-4 \tilde c}}{\lambda^{2N+1}} \,
	\int_{\mathbb{R}^{2N-1}}e^{-\mathcal{H}_{[-1,2N-1]}(\varphi)} 
	\,\prod_{i=1}^{2N-1}\,\dd\varphi_i\,.
\end{align*}
Recalling Proposition~\ref{PropLawOfW} and Proposition~\ref{PropBoundOnY}, we obtain
\begin{align*}
	\varphi^{(0,0)}_{2N+1}(0,0) & \;\le\; 
	\frac{e^{-4 \tilde c}}{\lambda^{2}} \,\int_{\mathbb{R}^{2}} 
	\frac{v(0)}{v(\varphi_{2N-1} - \varphi_{2N-2})} \,
	\p^{(0,0)} (W_{2N-2} \in  \dd\varphi_{2N-2},\ W_{2N-1} \in \dd\varphi_{2N-1}) \\
	& \;=\; \frac{v(0)}{\lambda^{2}}\,e^{-4 \tilde c} \, \e^{(0,0)}
	\bigg( \frac{1}{v(Y_{2N-1})} \bigg) \;\le\; \frac{v(0)}{\lambda^{2}}\,e^{-4 \tilde c}
	\, C \;=\; (const.) \,,
\end{align*}
which completes the proof of \eqref{eq:boundphi}.
\end{proof}


\medskip

\section{A lower bound on the partition function}

\label{sec:lb}

We are going to give an explicit lower bound on the partition function in
terms of a suitable renewal process. First of all, we rewrite
equation \eqref{eq:Z} as
\begin{equation} \label{eq:Z2}
	 \Ze = \sum_{k=0}^{N-1}\varepsilon^k\sum_{\substack{A\subseteq\{1,...,N-1\}\\|A|=k}}
	 \int e^{-\mathcal{H}_{[-1,N+1]}(\varphi)}
	 \prod_{m\in A}\delta_0(\dd\varphi_m)\,\prod_{n\in A^c}\dd\varphi_n \,,
\end{equation}
where we set $A^c := \{1, \ldots, N-1\} \setminus A$ for convenience.


\subsection{A renewal process lower bound}

We restrict the summation over $A$
in \eqref{eq:Z2} to the class of subsets $\mathfrak{B}_{2k}$ consisting of $2k$ points
organized in $k$ consecutive couples:
\[
	\mathfrak{B}_{2k} \;:=\; \big\{\{t_1-1,t_1, \ldots ,t_k-1,t_k\}\ |\ 0=t_0<t_1< \ldots <t_k \le N-1 
	\text{ and } t_i-t_{i-1}\geq 2 \ \forall i \big\} \,.
\]
Plainly, $\mathfrak{B}_{2k} = \emptyset$ for $k > (N-1)/2$. We then obtain from \eqref{eq:Z2}
\begin{align}\label{ZlowerBound}
	\Ze & \;\ge\; \sum_{k=0}^{\lfloor (N-1)/2 \rfloor}
	\varepsilon^{2k}\sum_{A\in\mathfrak{B}_{2k}}\int 
	e^{-\mathcal{H}_{[-1,N+1]}(\varphi)}\prod_{m\in A}\delta_0(\dd\varphi_m)
	\,\prod_{n\in A^C}\dd\varphi_n \nonumber \\
	&\;=\; \sum_{k=0}^{\lfloor (N-1)/2 \rfloor} \varepsilon^{2k}
	\sum_{\substack{0=t_0<t_1<...<t_k < t_{k+1} = N+1\\
	t_i-t_{i-1}\geq 2 \ \forall i \le k+1}}\ \prod_{j=1}^{k+1}\widetilde{K}(t_j-t_{j-1}) \,,
\end{align}
where we have set for $n \in \N$
\begin{equation} \label{eq:Ktilde}
	\widetilde{K}(n) \;:=\; \begin{cases}
	0 & \text{ if } n=1 \\
	\rule{0pt}{1.4em}e^{-\mathcal{H}_{[-1,2]}(0,0,0,0)} = e^{-2V_1(0) - 2V_2(0))} & \text{ if }n=2\\
	\rule{0pt}{3em}
	\!\!\left.\begin{split}
	& \int_{\mathbb{R}^{n-2}}e^{-\mathcal{H}_{[-1,n]}(w_{-1},...,w_n)}\dd w_1\cdots \dd w_{n-2}\\
	& \rule{0pt}{1.1em}\text{with } w_{-1}=0,w_0=0,w_{n-1}=0,w_n=0
	\end{split}\right\}
	& \text{ if }n\geq 3
    \end{cases} \ .
\end{equation}
We stress that a factorization of the form \eqref{ZlowerBound} is possible
because the Hamiltonian $\mathcal{H}_{[-1,N+1]}(\varphi)$ 
consists of two- and three-body terms and we have
restricted the sum over subsets in $\mathfrak{B}_{2k}$, that consist
of consecutive couples of zeros. We also note that the condition
$t_i-t_{i-1}\geq 2$ is immaterial, because by definition $\widetilde K(1)=0$.

We now give a probabilistic interpretation to the right hand side of \eqref{ZlowerBound}
in terms of a renewal process. To this purpose, for every $\gep > 0$
and for $n\in\N$ we define
\[
	K_\gep(1) \;:=\; 0 \,, \qquad
	K_\gep (n) \;:=\; \frac{\varepsilon^2}{\lambda^n}\,\widetilde{K}(n)\,e^{-\mu_\varepsilon n}
	\;=\; \varepsilon^2 \, \varphi_n^{(0,0)}(0,0) \, e^{-\mu_\varepsilon n} \,,
	\ \forall n \ge 2 \,.
\]
where the second equality follows recalling \eqref{eq:Ktilde}, Proposition~\ref{PropLawOfW}
and the definition \eqref{eq:phi} of the density $\phi_n$.
The constant $\mu_\gep$ is chosen to make $K_\gep$ a probability on $\N$:
\begin{equation} \label{eq:mu}
	\sum_{n\in\mathbb{N}} K_{\varepsilon}(n) \;=\;1 \,, \qquad
	\text{that is} \qquad \sum_{n=2}^\infty \varphi_n^{(0,0)}(0,0) \, e^{-\mu_\varepsilon n}
	\;=\; \frac{1}{\gep^2} \,.
\end{equation}
It follows from Proposition~\ref{PropLowerUpperBound} that $0 < \mu_\varepsilon < \infty$
for every $\varepsilon>0$.
We can therefore define a renewal process $(\{\eta_n\}_{n\ge 0}\}, \cP_\gep)$
on $\N_0$ with inter-arrival law $K_\gep(\cdot)$. More explicitly,
$\eta_0 := 0$ and the increments $\{\eta_{k+1}-\eta_k\}_{k \ge 0}$
are independent, identically distributed random variables with marginal law
$\mathcal{P}_{\varepsilon}(\eta_{k+1}-\eta_k = n) = K_{\varepsilon}(n)$.
Coming back to \eqref{ZlowerBound}, we can write
\begin{align}\label{ZEpsilonRenewal}
	\Ze & \;\ge\; \frac{\lambda^{N+1}\,e^{(N+1)\mu_\varepsilon}}{\varepsilon^2}
	\sum_{k=0}^{\lfloor (N-1)/2 \rfloor}
	\sum_{0=t_0<t_1<...<t_k < t_{k+1} = N+1}\ \prod_{j=1}^{k+1} K_\gep(t_j-t_{j-1})
	\nonumber \\
	& \;=\; \frac{\lambda^{N+1}\,e^{(N+1)\mu_\varepsilon}}{\varepsilon^2}
	\sum_{k=0}^{\lfloor (N-1)/2 \rfloor}
	\sum_{0=t_0<t_1<...<t_k < t_{k+1} = N+1}
	\cP_\gep \big( \eta_1 = t_1 , \ldots, \eta_{k+1} = t_{k+1} \big) \nonumber \\
	& \;=\; \frac{\lambda^{N+1}\,e^{(N+1)\mu_\varepsilon}}{\varepsilon^2}
	\sum_{k=0}^{\lfloor (N-1)/2 \rfloor} \cP_\gep \big( \eta_{k+1}=N+1 \big)
	\;=\; \frac{\lambda^{N+1}\,e^{(N+1)\mu_\varepsilon}}{\varepsilon^2}
	\cP_\gep \big( N+1 \in \eta \big) \,,
\end{align}
where in the last equality we look at $\eta = \{\eta_k\}_{k\ge 0}$ as a random
subset of $\N_0$, so that $\{N+1 \in \eta\} =
\bigcup_{m =1}^\infty \{\eta_m = N+1\}$ (note that $\cP_\gep(\eta_{k+1}=N+1)=0$ for 
$k > \lfloor (N-1)/2 \rfloor$).

We have thus obtained a lower bound on the partition function $\Ze$ of our
model in terms of the renewal mass function (or Green function) of the renewal
process $(\{\eta_n\}_{n\ge 0}\}, \cP_\gep)$.


\subsection{Proof of Theorem~\ref{Theorem}}

Recall the free energy from definition \ref{freeEnergy}
\begin{equation*}
	F(\varepsilon)=\lim_{N\rightarrow\infty}\frac{1}{N}\log{\frac{\Ze}{\cZ_{0,N}}}\ .
\end{equation*}
From now on, the limits $N\rightarrow\infty$ will be implicitly taken along the
odd numbers. Observe that by Proposition \ref{PropFreeModel} and both bounds in 
Proposition~\ref{PropLowerUpperBound} 
\[
	\lim_{N\rightarrow\infty}\frac{1}{N}\log{\cZ_{0,N}}
	\;=\; \lim_{N\rightarrow\infty}\frac{1}{N}\left((N+1)\log{\lambda}
	+\log{\varphi_{N+1}^{(0,0)}(0,0)}\right) \;=\; \log{\lambda}\ .
\]
Therefore for every $\varepsilon>0$ by \eqref{ZEpsilonRenewal} we obtain
\begin{align}
	\lim_{N\rightarrow\infty}\frac{1}{N}\log{\frac{\Ze}{\cZ_{0,N}}}
	& \;\geq\; \limsup_{N\rightarrow\infty}\frac{1}{N}
	\log{\left[\frac{\lambda^{N+1}\,e^{\mu_\varepsilon\,(N+1)}}{\varepsilon^2}\,
	\cP_\varepsilon(N+1\in\eta)\right]} -\, \log \lambda \nonumber\\
	\label{eq:quasilast}
	&\;\geq\; \mu_\varepsilon \,+\, \limsup_{N\rightarrow\infty}
	\frac{1}{N}\log \cP_\varepsilon(N+1\in\eta)  \,.
\end{align}
Since $\mathcal{P}_\varepsilon(\eta_1=n)>0$ for all $n\in\mathbb{N}$
with $n\geq 2$, the renewal process $(\{\eta_k\}_{k\ge 0}, \cP_\gep)$ is aperiodic and by the
classical renewal theorem $\cP_\varepsilon(N+1\in\eta) \to \frac{1}{m_\varepsilon}$
as $N\to\infty$, where
\[
	m_\varepsilon \;=\; \sum_{n\geq 2}n\,K_\varepsilon(n)
	\;=\; \varepsilon^2\sum_{n\geq 2}n\,\varphi_n^{(0,0)}(0,0)\,e^{-\mu_\varepsilon n}\ <\ \infty\ .
\]
by Proposition \ref{PropLowerUpperBound}. Therefore from \eqref{eq:quasilast}
we get $F(\gep) \ge \mu_\gep$.
As we already mentioned above, we have $\mu_\varepsilon>0$,
hence $F(\gep) > 0$, for all $\varepsilon>0$. This shows that our model
exhibit a trivial phase transition.\qed

\begin{rem}\rm
\label{th:improving}
We have just shown that $F(\gep) \ge \mu_\gep$. Recalling the definition \eqref{eq:mu} of $\mu_\gep$,
it is clear that the lower bound in \eqref{eq:boundphi} on $\phi_N^{(0,0)}(0,0)$ 
yields a corresponding lower bound on $\mu_\gep$, hence on $F(\gep)$.
Unfortunately, this lower bound is very poor: in fact, by standard Tauberian theorems, 
from \eqref{eq:boundphi} we get $\mu_\gep \ge \exp(-(const.)/\gep^2)$, 
which vanishes as $\gep \downarrow 0$
faster than any polynomial. On the other hand,
the conjectured correct behavior of the free energy, in analogy with the
purely gradient case, should be $F(\gep) \sim (const.)\, \gep^2$.

One could hope to sharpen the lower bound on $\mu_\gep$ by improving
the one on $\phi_N^{(0,0)}(0,0)$. This is possible, but only to a certain extent:
even the conjectured sharp lower bound 
$\phi_N^{(0,0)}(0,0) \ge (const.)/\sqrt{N}$ (in analogy with the purely gradient case)
would yield only $\mu_\gep \ge (const.) \, \gep^4$.
This discrepancy is a limitation of our lower bound technique: 
in order to have a genuine renewal structure,
the chain is forced to visit the defect line at \emph{couples of neighboring points},
which are rewarded $\gep^2$ instead of $\gep$. If one could replace
$1/\gep^2$ by $1/\gep$ in \eqref{eq:mu}, the lower bound 
$\phi_N^{(0,0)}(0,0) \ge (const.)/\sqrt{N}$ would yield
$\mu_\gep \ge (const.') \, \gep^2$, as expected.
\end{rem}


\appendix

\medskip

\section{Some recurrence properties}
\label{sec:Harris}

We have already remarked that $Y = \{Y_i\}_{i\in\N_0}$
is $\Leb$-irreducible, hence it is also $\pi$-irreducible,
see \eqref{eq:pi}, because $\pi$ is absolutely continuous with respect to $\Leb$.
By Proposition~4.2.2 in~\cite{cf:MT}, a maximal irreducibility measure
for $Y$ is
$\psi(\dd x) := \sum_{n=0}^\infty \frac{1}{2^{n+1}} (\pi \cP^n)(\dd x)$,
where we set $(\pi \cQ)(\dd x) := \int_{z \in \R} \pi(\dd z) \cQ(z, \dd x)$ 
for any kernel $\cQ$ and we use the standard notation $\cP^{0}(z, \dd x) := \gd_z(\dd x)$,
$\cP^{1} = \cP$ (we recall \eqref{eq:p}) and for $n \ge 1$
\begin{equation*}
	\cP^{n+1}(z, \dd x) \;:=\; \int_{y \in \R} \cP^{n} (z, \dd y) \, \cP(y, \dd x) \,.
\end{equation*}
Since the law $\pi$ is invariant for the kernel $\cP$, see \eqref{eq:invpi},
we have $\pi \cP^n = \pi$ for all $n\ge 0$, therefore the maximal irreducibility
measure $\psi$ is nothing but $\pi$ itself. Since a maximal irreducibility measure
is only defined up to equivalent measures (in the sense of Radon-Nikodym), it
follows that $\Leb$, which is equivalent to $\pi$, is a
maximal irreducibility measure.

(As a matter of fact, it is always true that if a $\phi$-irreducible 
Markov chain admits an invariant measure $\pi$,
then $\pi$ is a maximal irreducibility measure, cf. Theorem~5.2 in~\cite{cf:Num}.)

\smallskip

Next we prove that $Y$ is a $T$-chain, as it is
defined in Chapter~6 of~\cite{cf:MT}. To this purpose,
we first show that $Y$ is a Feller chain, that is,
for every bounded and continuous function $f: \R \to \R$ the function
$(\cP f)(x) := \int_\R \cP(x, \dd y) \, f(y)$ is bounded and continuous.
We recall that the function $v$ is continuous, as we have shown in \S\ref{sec:asymp}.
We then write
\begin{align*}
	(\cP f)(x) & \;:=\; \int_\R \cP(x, \dd y) \, f(y)
	\;=\; \frac{1}{\lambda \, v(x)} \, \int_\R
	e^{-V_1(y) - V_2(y-x)} \, v(y) \, f(y) \, \dd y \\
	& \;=\; \frac{1}{\lambda \, v(x)} \big( e^{-V_2} *
	(e^{-V_1} \, v \, f) \big)(x) \,,
\end{align*}
from which the continuity of $\cP f$ follows, because
$e^{-V_2} \in L^\infty(\R)$ and $(e^{-V_1} \, v \, f) \in L^1(\R)$
and we recall that the convolution of a function in $L^\infty(\R)$ with a function in $L^1(\R)$
is bounded and continuous.
Since $Y$ is a $\Leb$-irreducible Feller chain,
it follows from Theorem~6.0.1 (iii) in~\cite{cf:MT} that $Y$
is a $\Leb$-irreducible $T$-chain.

\,

Finally, we observe that from the drift condition \eqref{eq:drift}
it follows that $Y$ is a Harris recurrent chain. For this it suffices
to apply Theorem~9.1.8 in~\cite{cf:MT}, observing that the function
$U$ defined in \eqref{eq:U} is \emph{coercive}, i.e. $\lim_{|x| \to\infty}
U(x) = +\infty$, hence it is ``unbounded off petite sets''
(cf. \cite[\S8.4.2]{cf:MT}) because every compact set is petite for
irreducible $T$-chains, by Theorem~6.0.1 (ii) in~\cite{cf:MT}.


\section*{Acknowledgements}

We thank Jean-Dominique Deuschel for fruitful discussions.
F.C. gratefully acknowledges the support of the University
of Padova under grant CPDA082105/08.


\end{document}